\documentclass[11pt,draft]{amsart}
\usepackage{amssymb,amsmath}
\usepackage{mathrsfs}
\usepackage{graphicx}
\usepackage{color}
\usepackage{indentfirst}
\usepackage{enumerate}
\usepackage{multicol}
\usepackage{dsfont}

\usepackage[a4paper,left=2.0cm,right=2.0cm,top=1.0cm,bottom=1.6cm]{geometry}
\usepackage[hidelinks]{hyperref}
\newtheorem*{theorem}{Theorem}

\numberwithin{equation}{section}

\begin{document}

\title{A simple proof of the Lebesgue decomposition theorem}

\author{Tam\'as Titkos}
\address{Tam\'as Titkos, Department of Applied Analysis, E\"otv\"os Lor\'and University, P\'azm\'any P\'eter s\'et\'any 1/c, H-1117, Budapest, Hungary; }
\email{titkos@cs.elte.hu}


\maketitle
The aim of this short note is to present an elementary, self-contained, and direct proof for the classical Lebesgue decomposition theorem. In fact, I will show that the absolutely continuous part just measures the squared semidistance of the characteristic functions from a suitable subspace.

This approach also gives a decomposition in the finitely additive case, but it differs from the Lebesgue-Darst decomposition \cite{raex}, because the involved absolute continuity concepts are different.\\

\noindent\textbf{Notations.} Let $\mathcal{A}$ be a $\sigma$-algebra over $X\neq\emptyset$, and consider the finite measures $\mu,\nu:\mathcal{A}\to\mathbb{R}_+$ on it. The measure $\mu$ is $\nu$-absolutely continuous ($\mu\ll\nu$, in symbols) if $\nu(A)=0$ implies $\mu(A)=0$ for all $A\in\mathcal{A}$. Singularity of $\mu$ and $\nu$ (denoted by $\mu\perp\nu$) means that the only measure dominated by both $\mu$ and $\nu$ is the zero measure. As it is known, this is equivalent with the existence of a measurable set $P\in\mathcal{A}$ such that $\mu(P)=\nu(X\setminus P)=0$.

\begin{theorem} Let $\mu$ and $\nu$ be finite measures on $\mathcal{A}$. Then $\mu$ splits uniquely into $\mu_{\mathrm{ac}}\ll\nu$ and $\mu_{\mathrm{s}}\perp\nu$.
\end{theorem}
\begin{proof}
Consider the real vector space $\mathscr{E}$ of real valued $\mathcal{A}$-measurable step-functions and let $\mathscr{N}$ be the linear subspace generated by the characteristic functions of those measurable sets $A$ such that $\nu(A)=0$. Define the set function $\mu_{\mathrm{ac}}$ by
\begin{align*}
\mu_{\mathrm{ac}}(A):=\inf\limits_{\psi\in\mathscr{N}}\int\limits_X|\mathds{1}_A-\psi|^2~\mathrm{d}\mu\hspace{0.5cm}(A\in\mathcal{A}).
\end{align*}
It is clear that $\mu_{\mathrm{ac}}\leq\mu$ ($\psi:=\mathds{1}_{\emptyset}$), and that $\nu(A)=0$ implies $\mu_{\mathrm{ac}}(A)=0$ ($\psi:=\mathds{1}_A$). Furthermore, trivial verification shows that if $A$ and $B$ are disjoint elements of $\mathcal{A}$, then
\begin{align*}
\inf_{\psi\in\mathscr{N}}\int\limits_X |\mathds{1}_{{}_{A\cup B}}-\psi|^2~\mathrm{d}\mu
=\inf_{\psi\in\mathscr{N}}\int\limits_X|\mathds{1}_{{}_{A}}-\psi|^2~\mathrm{d}\mu+\inf_{\psi\in\mathscr{N}}\int\limits_X |\mathds{1}_{{}_{B}}-\psi|^2~\mathrm{d}\mu.
\end{align*}
Since $\mu_{\mathrm{ac}}$ is nonnegative, additive, and dominated by the measure $\mu$, we infer that $\mu_{\mathrm{ac}}$ is a measure itself.

What is left is to show that $\mu_{\mathrm{s}}:=\mu-\mu_{\mathrm{ac}}$ and $\nu$ are singular, and that the decomposition is unique. Both follow immediately from the fact that $\mu_{\mathrm{ac}}$ is maximal among those measures $\vartheta$ such that $\vartheta\leq\mu$ and $\vartheta\ll\nu$. Indeed, let $\vartheta$ be such a measure, $\psi\in\mathscr{N}$, and observe that
\begin{align*}
\vartheta(A)=\int\limits_X|\mathds{1}_A|^2~\mathrm{d}\vartheta=\int\limits_X|\mathds{1}_A-\psi|^2~\mathrm{d}\vartheta\leq\int\limits_X|\mathds{1}_A-\psi|^2~\mathrm{d}\mu.
\end{align*}
Taking the infimum over $\mathscr{N}$ we obtain that $\vartheta\leq\mu_{\mathrm{ac}}$.

Now, let $\eta$ be a measure, such that $\eta\leq\nu$ and $\eta\leq\mu-\mu_{\mathrm{ac}}$. In this case, $\mu_{\mathrm{ac}}+\eta\leq\mu$ and $\mu_{\mathrm{ac}}+\eta\ll\nu$, thus $\eta=0$. If $\mu=\mu_1+\mu_2$, where $\mu_1\ll\nu$ and $\mu_2\perp\nu$, then  $\mu_{\mathrm{ac}}-\mu_1$ is a measure, which is simultaneously $\nu$-absolutely continuous and $\nu$-singular. This yields that $\mu_1=\mu_{\mathrm{ac}}$.
\end{proof}

\end{document}